\UseRawInputEncoding
\documentclass[11pt, twoside]{article}
\usepackage{amsfonts}
\usepackage{amsmath,amssymb}
\usepackage{multicol}
\usepackage{color}
\usepackage{cite}
\newtheorem{theorem}{Theorem}[section]
\newtheorem{lemma}[theorem]{Lemma}
\newtheorem{corollary}[theorem]{Corollary}

\newtheorem{definition}[theorem]{Definition}
\newtheorem{remark}[theorem]{Remark}

\numberwithin{equation}{section}
\newenvironment{proof}[1][Proof]{\noindent\textbf{#1. }}{\hfill $\Box$}
\allowdisplaybreaks

\makeatletter
\begin{document}
\author{Quanguo Zhang \thanks{Corresponding author:
zhangqg07@163.com.}
\\
{\small Department of Mathematics, Luoyang Normal University,
 Luoyang, Henan 471022, P.R. China}}
\title{\textbf{\Large  On the critical exponents  for a  fractional diffusion-wave equation with a nonlinear memory term in a bounded domain
  } }
\date{}
\maketitle
\begin{abstract}
In this paper, we prove sharp blow-up and global existence results for a time 
fractional diffusion-wave equation with a nonlinear memory term in a bounded domain, where the fractional derivative in time is
taken in the sense of Caputo type. Moreover, we also give a result for  nonexistence of global solutions to a wave equation with a nonlinear memory term in a bounded domain.  The proof of blow-up results is based on the eigenfunction method and the asymptotic properties of solutions for an ordinary fractional differential inequality.


\end{abstract}

\section{Introduction}

\noindent

This paper  is mainly concerned with the blow-up and global existence of solutions for the following time fractional diffusion-wave equation with a nonlinear memory term:
\begin{equation}\label{20.1}
\left\{\begin{array}{l}{_0D_t^\alpha u}-\triangle u={}_0^{}I_t^{\gamma}(|u|^{p}),\ \
(t,x)\in (0,T)\times\Omega,\\
u(t,x)=0,\ \ (t,x)\in (0,T)\times\partial\Omega,\\
u(0,x)=u_0(x),\ \  u_t(0,x)=u_1(x),\ \  x\in \Omega,
\end{array}\right.
\end{equation}
where $\Omega$ is a bounded domain in $\mathbb{R}^N$ with smooth boundary $\partial\Omega$, $u_0,u_1\in L^\infty(\Omega)$, $1<\alpha<2$, $\gamma>0$, $p>1$, and
\[{_0D_t^\alpha u}=\frac{\partial^2}{\partial
t^2}[{_0I_t^{2-\alpha}(u(t,x)-u_1(x)t-u_0(x))}].\]
We also give a result for  nonexistence of global solutions to problem \eqref{20.1} with $\alpha=2$.

In recent years, fractional differential equations have gained considerable popularity and importance, due to their demonstrated applications in seemingly widespread fields of science and engineering \cite{Metzler1,kilbas2006}. Hence, recently, there are
a lot of papers on the existence and properties of solutions for  fractional differential equations\cite{vr,kim,Schneider,zacher,ZhangSun,acv,R.N.Wang,wzf,llw,zhou,zhangli2,
zhangli3,Zhangli4, tuan,giga,asofwa,andrade,agkw}.

Equation \eqref{20.1} interpolates the heat equation and the wave equation. Let us present a historical overview on some blow-up and global existence results for semilinear heat and wave equations.
For the following Cauchy problem of semilinear heat equation
\begin{equation}\label{20.4}
\left\{\begin{array}{l} u_t-\triangle u=|u|^{p-1}u,\ \ x\in \mathbb{R}^N,\ \ t>0,\\
u(0,x)=u_0(x),\ \  x\in \mathbb{R}^N,
\end{array}\right.
\end{equation}
it is well known that all solutions of \eqref{20.4} with $u_0\geq 0$, $u_0\not\equiv0$, blow up in finite time if and only if $p\leq 1+\frac{2}{N}$, and if $p>1+\frac{2}{N}$ and $\|u_0\|_{L^{q_c}(\mathbb{R}^N)}$ is small enough, where $q_c=\frac{N(p-1)}{2}$, then the solution of \eqref{20.4} is global. The number $1+\frac{2}{N}$ is called the Fujita critical exponent of problem \eqref{20.4}.
We refer to \cite{qs} for the proof of these results.
Recently, Zhang and Sun\cite{ZhangSun} considered problem \eqref{20.4} with ${_0D_t^\alpha u}$ $ (0<\alpha<1)$ instead of $u_{t}$.  They proved that if $1 < p < 1+\frac2N$, then any nontrivial positive solution  blows up in finite time, while if $p \geq 1+\frac2N$ and the initial value is sufficiently small, the solution exists globally.

For the semilinear
wave equation
\begin{equation}\label{202.4}
\left\{\begin{array}{l} u_{tt}-\triangle u=|u|^{p},\ \ x\in \mathbb{R}^N,\ \ t>0,\\
u(0,x)=u_0(x),\ \  u_t(0,x)=u_1(x),\ \ x\in \mathbb{R}^N,
\end{array}\right.
\end{equation}
the critical exponent is $p_c(N)$, which is the positive root of the quadratic equation $(N-1)p^2-(N+1)p-2=0 $ for $N>1$, see \cite{Yordanov} and the references therein. Recently, in \cite{Zhangli4}, the authors determined the critical exponents of problem \eqref{202.4} with ${_0D_t^\alpha u}$ $ (1<\alpha<2)$ instead of $u_{tt}$  when $u_1\equiv0$ and $u_1\not\equiv0, $ respectively.

Let us now turn to the study of semilinear time fractional diffusion equations with a nonlinear memory term and semilinear wave equations with  a nonlinear memory term.   There have been many papers that considered the existence and nonexistence of  global solutions for these problems \cite{llz,T.Cazenave,fino1,fino2,chen,chen2,Abbicco,li,zhangli2,zhangli3}.

In \cite{T.Cazenave}, Cazenave et al. studied the following semilinear heat equation with a nonlinear memory term
\[
u_t-\triangle u={}_0^{}I_t^{1-\gamma}(|u|^{p-1}u)
\]
on both $\mathbb{R}^N$ and a bounded domain $\Omega\subset\mathbb{R}^N$, and obtained the critical exponents of this problem. Recently, the authors of \cite{li,zhangli2,zhangli3} generalized the results of \cite{T.Cazenave} to the time fractional case ($0<\alpha<1$), and gave the critical exponents of this problem  when $\alpha<\gamma$ and $\alpha\geq\gamma$, respectively.

In \cite{chen}, Chen and Palmieri established a generalized Strauss exponent $p_0(N,\gamma)$
for problem \eqref{20.1} with $\alpha=2$ on $\mathbb{R}^N$. They proved blow-up of energy solutions  if $1<p\leq p_0(N,\gamma)$ for $N\geq 2$, and $p>1$ for $N=1$. Fino and Jazar \cite{fino2} proved that the nontrivial solution of problem \eqref{20.1} with $\alpha=2$ blows up in finite time if $p(1-\gamma)<1$. In \cite {llz}, the authors proved all solutions of \eqref{20.1} with $\alpha=2$ on $(0,\infty)$ blow up, provided that $u_0$, $u_1$ have compact support and satisfy a certain positivity condition.

Motivated by the above results, in this paper, we investigate  sharp blow-up and global existence of  solutions for problem \eqref{20.1}, and then extend the results in \cite{T.Cazenave,zhangli3,zhangli2} to the time fractional case ($1<\alpha<2$).  Moreover, we also give a result for  nonexistence of global solutions to a wave equation with a nonlinear memory term(i.e. \eqref{20.1} with $\alpha=2$).

It should be mentioned that the study of the blow-up and global existence of solutions  for \eqref{20.1} is not a simple generalization of those in the previous researches on semilinear diffusion equations with a nonlinear memory term. For one thing, our proof of blow-up results is based on the asymptotic properties of solutions for an ordinary fractional differential inequality. In the case $0<\alpha<1$, the proof of these properties depends on the nonnegativity of the Mittag-Leffler function $E_{\alpha,\alpha}(t)$ for $t\in \mathbb{R}$. But, $E_{\alpha,\alpha}(t)$ is not nonnegative on $(-\infty,0]$ in the case $1<\alpha<2$, which gives us some technical difficulties in the treatments.  To overcome these difficulties, we study the asymptotic behavior of $_0I_t^\beta w$ for some $\beta>0$ instead of $w$ (see Lemma \ref{lemma17}, Corollary \ref{coro1}-\ref{coro2}). This also allows us to obtain a result of the nonexistence of global solutions for problem \eqref{20.1} with $\alpha=2$, which asserts that the case $p(1-\gamma)=1$ is in the blow-up category.  For another, the estimates of the solution operators on $L^\infty(\Omega)$ are crucial to prove the global existence of solutions. In the case $0<\alpha<1$, one can easily obtain estimates of the solution operators on $L^\infty(\Omega)$, since the solution operators can be represented by a probability density function and the heat semigroup in $\Omega$ with
the Dirichlet boundary condition.  But, for the case $1<\alpha<2$, this representation is invalid. We prove the estimates of the solution operators on $L^\infty(\Omega)$ by the complex integral representations of the solution operators (see Lemma \ref{lemma100}).

The remaining part of the paper is organized as follows.  In Section 2, we present some
results on the Mittag-Leffler function, the fractional derivatives and the
fractional integrals, and show some results on an ordinary fractional differential inequality  which will be used to prove the blow-up results. In Section 3, the local existence and uniqueness of the mild solution of problem \eqref{20.1} are given, and we prove sharp blow-up and global existence of solutions for problem \eqref{20.1}.

For simplicity, in this paper, we use $C$ to denote a positive constant which may vary from line to line, but it is not essential to the analysis of the problem.

\section{Preliminaries}

\noindent

In this section, we present some preliminaries that will be used in
the next section.

First, we recall some properties of the Mittag-Leffler function with two parameters \cite{kilbas2006,podlubny1999}. The Mittag-Leffler function with two parameters is defined by
\[
E_{\alpha,\beta}(z)=\sum_{k=0}^\infty\frac{z^k}{\Gamma(\alpha k+\beta)},
\ \alpha,\beta\in\mathbb{C},\alpha>0,\ \ E_{\alpha}(z)=E_{\alpha,1}(z),\ \ z\in \mathbb{C},
\]
which is an entire function. Let $\mu$ be a real number such that
$\frac{\pi\alpha}{2}<\mu<\min\{\pi,\pi\alpha\}.$ $E_{\alpha,\beta}(z)$ has different asymptotic behavior at infinity for $0<\alpha<2$ and $\alpha=2$. If $0<\alpha<2$, then for $N\in \mathbb{N}$,
\begin{equation}\label{100.1}
E_{\alpha,\beta}(z)=-\sum_{k=1}^N\frac{1}{\Gamma(\beta-\alpha k)}\frac{1}{z^k}+O(\frac{1}{z^{N+1}})
\end{equation}
with $|z|\rightarrow\infty,$ $\mu\leq |\arg(z)|\leq \pi$. If $\alpha=2$, then for $x>0$ and $N\in \mathbb{N}$,
\begin{equation}\label{100.2}
E_{2,\beta}(-x)=x^{\frac{1-\beta}{2}}\cos(\sqrt{x}+\frac{\pi(1-\beta)}{2})
-\sum_{k=1}^N\frac{1}{\Gamma(\beta-2 k)}\frac{(-1)^k}{x^k}+O(\frac{1}{x^{N+1}})
\end{equation}
with $x\rightarrow+\infty.$

Next, we recall the definitions of the fractional
derivatives and fractional integrals \cite{kilbas2006}.
For $T>0$, $\alpha>0$, the Riemann-Liouville fractional integrals are defined
by
\[
{_0I_t^{\alpha}}u=\frac{1}{\Gamma(\alpha)}\int_0^t\frac{u(s)}{(t-s)^{1-\alpha}}ds,\ \ {_tI_T^{\alpha}}u=\frac{1}{\Gamma(\alpha)}\int_t^T\frac{u(s)}{(s-t)^{1-\alpha}}ds.
\]
For $T>0$, $\alpha\in (1,2)$, the Caputo fractional derivatives are defined
by
\[
{}_0D_t^\alpha g= \frac{d^2}{dt^2}{}_0I_t^{2-\alpha}[ g(t)-g'(0)t-g(0)],\ _tD_T^\alpha g= \frac{d^2}{dt^2}{}_tI_T^{2-\alpha}[ g(t)-g'(T)(t-T)-g(T)].
\]
When $\alpha=2$, we define ${}_0D_t^2 g={}_tD_T^2 g=g''(t)$.

The fractional derivatives and fractional integrals have the following properties.
\begin{lemma}\label{l1}\cite{kilbas2006} Let $\alpha,\beta>0$ and $T>0$.
\begin{enumerate}
  \item [\rm (i)] ${}_0I_t^\alpha$ and ${}_tI_T^\alpha$ are bounded in $L^p(0,T)$ for every $p\geq 1$.
  \item [\rm (ii)] ${}_0I_t^\alpha({}_0I_t^\beta f)={}_0I_t^{\alpha+\beta}f$ and ${}_tI_T^\alpha({}_tI_T^\beta f)={}_tI_T^{\alpha+\beta}f$ if $f\in L^1(0,T)$.
  \item [\rm (iii)] If $f\in L^p(0,T)$, $g\in L^q(0,T)$ and $p,q\geq1$, $1/p+1/q=1$, then
\begin{equation}\label{112-4}
\int_0^T({_0I_t^\alpha}f)g(t)
dt=\int_0^T({_tI_T^\alpha}g)f(t)dt.
\end{equation}
\item [\rm (iv)] If $\alpha\in (1,2)$, $g\in AC^2([0,T])$, then ${}_0D_t^\alpha g$ and ${}_tD_T^\alpha g$ exist almost everywhere on $[0,T]$ and ${}_0D_t^\alpha g={}_0I_t^{2-\alpha}g''(t)$, ${}_tD_T^\alpha g={}_tI_T^{2-\alpha}g''(t)$.
\item [\rm (v)] For $\lambda\in \mathbb{C}$ and $\gamma> 0$, we have
    \begin{equation}\label{7.25.1}
    {}_0I_t^\gamma [t^{\beta-1}E_{\alpha,\beta}(\lambda t^\alpha)]=t^{\beta+\gamma-1}E_{\alpha,\beta+\gamma}(\lambda t^\alpha).
    \end{equation}
\end{enumerate}
\end{lemma}

For the fractional derivatives, we have the following formula of integration by parts.
\begin{lemma}\label{l2}
Assume that $1<\alpha\leq 2$, $g\in AC^2([0,T]) $, $g(T)=g'(T)=0$, $f\in C^1([0,T])$,  ${_0D_t^\alpha}f$ exists almost everywhere for $t\in (0,T)$ and ${_0D_t^\alpha}f\in L^1(0,T)$. Then
\begin{equation}\label{112-3}
\int_0^T{_0D_t^\alpha}f\cdot g
dt=\int_0^T(f(t)-f'(0)t-f(0))\ {_tD_T^\alpha}gdt.
\end{equation}
\end{lemma}
\begin{proof}
We only give the proof for the case $1<\alpha<2$, since the conclusion of the case $\alpha=2$ can be obtained by the classic formula of integration by parts.

Since $f\in C^1([0,T])$, a simple calculation shows that $\frac{d}{dt}[{}_0^{}I_t^{2-\alpha}(f(t)-f'(0)t-f(0))]={}_0^{}I_t^{2-\alpha}(f'(t)-f'(0))$. Hence the assumptions on $f$ and $g$ imply that
\[
[\frac{d}{dt}{}_0^{}I_t^{2-\alpha}(f(t)-f'(0)t-f(0))]g(t)\big|_0^T=0,
\]
and
\[
{}_0^{}I_t^{2-\alpha}[ f(s)-f'(0)s-f(0)]g'(t)\big|_0^T=0.
\]
Hence, using the definition of Caputo fractional derivatives, Lemma \ref{l1} and the classic formula of integration by parts, one gets
\begin{align*}
\int_0^T{_0D_t^\alpha}f\cdot g
dt&=\int_0^T \frac{d^2}{dt^2}{}_0^{}I_t^{2-\alpha}[ f(s)-f'(0)s-f(0)]g(t)dt\\
&=-\int_0^T \frac{d}{dt}{}_0^{}I_t^{2-\alpha}[ f(s)-f'(0)s-f(0)]g'(t)dt\\
&=\int_0^T {}_0^{}I_t^{2-\alpha}[ f(s)-f'(0)s-f(0)]g''(t)dt\\
&=\int_0^T [ f(t)-f'(0)t-f(0)]{}_t^{}I_T^{2-\alpha}g''(t)dt\\
&=\int_0^T[f(t)-f'(0)t-f(0)]\ {_tD_T^\alpha}gdt.
\end{align*}
This completes the proof of Lemma \ref{l2}.
\end{proof}

In order to prove blow-up results by the eigenfunction method due to \cite{Ka}, we prove the following results on an ordinary fractional differential equation, which are similar to ones in \cite{T.Cazenave, zhangli2}.
\begin{lemma}\label{lemma17}
Let $T>0$, $\gamma>0$, $1<\alpha<2$,  $p>1$, $a,b>0$ and $f\in L^1(0,T)$. If $w\in C^1([0,T])$, and $w\not \equiv0$ satisfies $_0I_t^{2-\alpha}(w-w'(0)t-w(0))\in AC^2([0,T])$ and for almost every $t\in [0,T]$,
\begin{equation}\label{1.7-1}
_0D_t^\alpha w+aw={}_0^{}I_t^{\gamma}f, \ \ f(t)\geq b|w(t)|^p,
\end{equation}
then $w$ satisfies the following properties:
\begin{enumerate}
  \item[\rm (i)] There exist positive constants $K_1$ and $K_2$ independent of $T$ such that \[Tw'(0)+w(0)\leq K_1T^{\alpha+\gamma-\frac{p\gamma}{p-1}}
      +K_2T^{-\frac{\alpha+\gamma}{p-1}}.\]
\item[\rm (ii)] If $T=+\infty,$  then $\liminf_{t\rightarrow +\infty}|w(t)|=0$ and $\liminf_{t\rightarrow +\infty}t^{\min\{\frac{\alpha+\gamma-1}{p},\frac{\gamma}{p-1}\}}|w(t)|<+\infty$.
\item[\rm (iii)] If $T=+\infty,$ $\gamma\geq \frac{\alpha}{2}$ and $\alpha+\gamma>2$,  then $\liminf_{t\rightarrow +\infty}t^{1-\gamma} w(t)>0$.
\item[\rm(iv)] If $T=+\infty,$ $\gamma\geq \frac{\alpha}{2}$, $\alpha+\gamma\leq 2$ and $w'(0)=0$, then $\liminf_{t\rightarrow +\infty}t^{1-\gamma} w(t)>0$.
\item[\rm(v)] If $T=+\infty,$ $\gamma\geq \frac{\alpha}{2}$, $\alpha+\gamma\leq 2$ and $w'(0)>0$, then $\liminf_{t\rightarrow +\infty}t^{\alpha-1} w(t)>0$.
\item[\rm(vi)] If $T=+\infty,$ $0<\gamma< \frac{\alpha}{2}$, then for every $\beta$ satisfying $\frac{\alpha}{2}\leq \gamma+\beta<1$ and $0<\beta<\alpha$, we have $\liminf_{t\rightarrow +\infty}t^{1-\gamma-\beta} {}_0I_t^\beta w>0$ if $\alpha+\gamma>2$, or $\alpha+\gamma\leq 2$ and $w'(0)=0$.
\item[\rm(vii)] If $T=+\infty,$ $0<\gamma< \frac{\alpha}{2}$, $\alpha+\gamma\leq 2$ and $w'(0)>0$, then for every $\beta$ satisfying $\frac{\alpha}{2}\leq \gamma+\beta<1$ and $0<\beta<\alpha$, we have $\liminf_{t\rightarrow +\infty}t^{\alpha-\beta-1} {}_0I_t^\beta w>0$.
\item[\rm(viii)] If one of the following conditions holds:
\begin{enumerate}
\item[\rm(a)] $\alpha+\gamma> 2$ and $p(1-\gamma)\leq 1$;
\item[\rm(b)] $\alpha+\gamma\leq 2$, $w'(0)=0$ and $p(1-\gamma)\leq 1$;
\item[\rm(c)] $\alpha+\gamma=2$, $w'(0)>0$ and $p(1-\gamma)\leq 1$;
\item[\rm(d)] $\alpha+\gamma<2$, $w'(0)>0$ and $p<1+\frac{\gamma}{\alpha-1}$,
\end{enumerate}
then $T<+\infty$.
\end{enumerate}
\end{lemma}
\begin{proof} The proofs of Property (i) and (ii) are similar to those of Lemma 5(i),(ii) and (iv) in \cite{zhangli2}. For the convenience of the reader and the completeness of the paper, here we give sketchy proofs of Property (i) and (ii).

(i) We use the test function method to prove this conclusion. From \eqref{112-4}, \eqref{112-3} and \eqref{1.7-1}, we know that
\begin{equation}\label{6.10-2}
\int_0^T[w({}_tD_T^\alpha \varphi)+aw\varphi]dt\geq b\int_0^T|w|^p{}_t^{}I_T^{\gamma}\varphi dt+w(0)\int_0^T{}_tD_T^\alpha \varphi dt+w'(0)\int_0^Tt\cdot{}_tD_T^\alpha \varphi dt,
\end{equation}
for every $\varphi\in AC^2([0,T])$ with $\varphi(T)=\varphi'(T)=0$, $\varphi\geq0$.

Note that for $l\geq \frac{p(\alpha+\gamma)}{p-1}$,
\begin{equation}\label{113-2}
{}_tD_T^\gamma (1-\frac{t}{T})^l=T^{-l}{}_tD_T^\gamma (T-t)^l =\frac{\Gamma(l+1)}{\Gamma(l+1-\gamma)}T^{-l}(T-t)^{l-\gamma},
\end{equation}
\begin{equation}\label{7.26.1}
{}_tI_T^\gamma[ {}_tD_T^\gamma (1-\frac{t}{T})^l]  =\frac{\Gamma(l+1)}{\Gamma(l+1-\gamma)}T^{-l}{}_tI_T^\gamma(T-t)^{l-\gamma}
=T^{-l}(T-t)^{l},
\end{equation}
\begin{equation}\label{113-3}
{}_tD_T^\alpha[{}_tD_T^\gamma (1-\frac{t}{T})^l]=\frac{\Gamma(l+1)}{\Gamma(l+1-\alpha-\gamma)}T^{-l}
(T-t)^{l-\alpha-\gamma},
\end{equation}
(see, e.g., Property 2.16 in \cite{kilbas2006}).
We take $\varphi(t)={}_tD_T^\gamma \psi_T$ as the test function in \eqref{6.10-2}, where $\psi_T(t)=(1-\frac{t}{T})^l$. It should be emphasized that choosing the test function of the type   $\psi_T(t)$ to prove the nonexistence of global solutions to fractional differential equations firstly appeared in \cite{Ki}. By \eqref{6.10-2}-\eqref{113-3}, we have
\begin{align*}
&w'(0)\int_0^Tt\cdot{}_tD_T^\alpha({}_tD_T^\gamma \psi_T)dt+w(0)\int_0^T{}_tD_T^\alpha({}_tD_T^\gamma \psi_T)dt+b\int_0^T|w|^p\psi_Tdt\\
&\leq \int_0^Tw[{}_tD_T^\alpha({}_tD_T^\gamma \psi_T)]dt+a\int_0^Tw({}_tD_T^\gamma \psi_T)dt\\
&\leq \frac{b}{2}\int_0^T|w|^p \psi_T dt
+CT^{1-\frac{p\gamma}{p-1}}+ CT^{1-\frac{p(\alpha+\gamma)}{p-1}}
\end{align*}
for some constant $C>0$, where we have used Young's inequality.
From this,  we know that there exist positive constants $K_1$ and $K_2$ such that
\begin{align}\label{6.16.1}
\int_0^T|w|^p \psi_T dt+w(0)T^{1-\alpha-\gamma}+w'(0)T^{2-\alpha-\gamma}\leq K_1T^{1-\frac{p\gamma}{p-1}}+ K_2T^{1-\frac{p(\alpha+\gamma)}{p-1}}.
\end{align}
Consequently,
\begin{align*}
Tw'(0)+w(0)\leq K_1T^{\alpha+\gamma-\frac{p\gamma}{p-1}}+K_2T^{-\frac{\alpha+\gamma}{p-1}}.
\end{align*}

(ii) We prove the first part in the statement.  We argue by contradiction. Suppose that there exist $\tau>0$ and $\eta>0$ such that $|w(t)|\geq \eta$ for all $t\geq\tau$. Then we deduce from \eqref{6.16.1} that for $T\geq 2\tau$,
\begin{align*}
\frac{T\eta^p}{2(l+1)}=\eta^p\int_{\frac{T}{2}}^T\psi_T(t)dt&\leq \int_\tau^T |w|^p\psi_Tdt\\
&\leq K_1T^{1-\frac{p\gamma}{p-1}}+K_2T^{1-\frac{p(\alpha+\gamma)}{p-1}}
+|w(0)|T^{1-(\alpha+\gamma)}+|w'(0)|T^{2-(\alpha+\gamma)}.
\end{align*}
In other words,
\[\frac{1}{2(l+1)}\eta^p\leq K_1T^{-\frac{p\gamma}{p-1}}+K_2T^{-\frac{p(\alpha+\gamma)}{p-1}}
+|w(0)|T^{-(\alpha+\gamma)}+|w'(0)|T^{1-(\alpha+\gamma)}.\]
Letting $T\rightarrow+\infty$, we have $\eta=0$ which contradicts $\eta>0$. Thus $\liminf_{t\rightarrow+\infty}|w(t)|=0$.

Next, we prove the second part in the statement.
In fact, since $\liminf_{t\rightarrow+\infty}|w(t)|=0$, we know that there exist $\tilde{\tau}\geq0$ and a nondecreasing sequence $\{t_n\}$ such that $t_n\rightarrow +\infty$ and $|w(t_n)|=\min_{\tilde{\tau}\leq t\leq t_n}|w(t)|$. Hence it follows from \eqref{6.16.1} that for $t_n\geq2\tilde{\tau}$,
\begin{align*}
|w(t_n)|^p\int_{\frac{t_n}{2}}^{t_n}\psi_{t_n}(t)dt&\leq \int_{\tilde{\tau}}^{t_n}|w(t)|^p\psi_{t_n}(t)dt\\
&\leq K_1{t_n}^{1-\frac{p\gamma}{p-1}}+K_2{t_n}^{1-\frac{p(\alpha+\gamma)}{p-1}}
+|w(0)|{t_n}^{1-(\alpha+\gamma)}+|w'(0)|{t_n}^{2-(\alpha+\gamma)}.
\end{align*}
This implies that
\begin{align*}
|w(t_n)|^p&\leq C[{t_n}^{-\frac{p\gamma}{p-1}}+{t_n}^{-\frac{p(\alpha+\gamma)}{p-1}}
+|w(0)|{t_n}^{-(\alpha+\gamma)}+|w'(0)|{t_n}^{1-(\alpha+\gamma)}]\\
&\leq C[{t_n}^{-\frac{p\gamma}{p-1}}+{t_n}^{1-(\alpha+\gamma)}]\leq C{t_n}^{-\min\{\alpha+\gamma-1,\frac{p\gamma}{p-1}\}}.
\end{align*}
Therefore $\liminf_{n\rightarrow\infty}t_n^{\min\{\frac{\alpha+\gamma-1}{p},\frac{\gamma}{p-1}\}}|w(t_n)|<+\infty$, which proves the desired conclusion.

(iii) In terms of $_0D_t^\alpha w+aw={}_0I_t^\gamma f(t)$, the solution $w$ is
explicitly expressed as follows (see e.g. \cite{zhou,kilbas2006,agkw})
\begin{align*}
w(t)&= E_\alpha(-at^\alpha)w(0)+tE_{\alpha,2}(-at^\alpha)w'(0)+
\int_0^t(t-s)^{\alpha-1}E_{\alpha,\alpha}(-a(t-s)^\alpha){}_0I_s^\gamma f ds.
\end{align*}
Then it follows from Fubini's theorem and \eqref{7.25.1} that
\begin{align}\label{8.11.1}
w(t)&= E_\alpha(-at^\alpha)w(0)+tE_{\alpha,2}(-at^\alpha)w'(0)+
\frac{1}{\Gamma(\gamma)}\int_0^t\int_0^{t-\tau}h^{\alpha-1}(t-\tau-h)^{\gamma-1}
E_{\alpha,\alpha}(-ah^\alpha)
 dh f(\tau)d\tau\nonumber\\
&= E_\alpha(-at^\alpha)w(0)+tE_{\alpha,2}(-at^\alpha)w'(0)+
\int_0^t(t-s)^{\gamma+\alpha-1}E_{\alpha,\alpha+\gamma}(-a(t-s)^\alpha)f(s) ds.
\end{align}
Furthermore, observing $E_{\alpha,\rho}(z)>0$ for all $z\in \mathbb{R}$ when $1<\alpha<2$ and $\rho\geq \frac{3\alpha}{2}$ ( see Theorem 2 in \cite{Pskhu}), we derive from \eqref{1.7-1} and \eqref{8.11.1}  that
\begin{equation}\label{18-6}
w(t)\geq E_\alpha(-at^\alpha)w(0)+tE_{\alpha,2}(-at^\alpha)w'(0)+b
\int_0^t(t-s)^{\alpha+\gamma-1}E_{\alpha,\alpha+\gamma}(-a(t-s)^\alpha)|w(s)|^p ds.
\end{equation}

Sine $w\not\equiv0$, we suppose $w(t_0)\not=0$ for some $t_0\in [0,T)$. Then there exists $\delta>0$ such that $w(t)\not=0$ for $t\in [t_0,t_0+\delta]$. In view of \eqref{100.1}, we have
\[
\lim_{t\rightarrow +\infty} t^\alpha E_{\alpha,\alpha+\gamma}(-at^\alpha)=\frac{1}{a\Gamma(\gamma)}.
\]
This implies
\[
\lim_{t\rightarrow +\infty} \int_{t_0}^{t_0+\delta}(t-\tau)^\alpha E_{\alpha,\alpha+\gamma}(-a(t-\tau)^\alpha)d\tau=\frac{\delta}{a\Gamma(\gamma)}.
\]
Hence, for $t$ large enough, we deduce from \eqref{18-6} that
\begin{align}\label{6.14.1}
w(t)\geq& E_\alpha(-at^\alpha)w(0)+tE_{\alpha,2}(-at^\alpha)w'(0)\nonumber\\
&+ b\min_{t\in [t_0,t_0+\delta]}|w(t)|^p \int_{t_0}^{t_0+\delta}(t-s)^{\alpha+\gamma-1}E_{\alpha,\alpha+\gamma}(-a(t-s)^\alpha) ds\nonumber\\
\geq& E_\alpha(-at^\alpha)w(0)+tE_{\alpha,2}(-at^\alpha)w'(0)+ \frac{b\delta t^{\gamma-1}}{2a\Gamma(\gamma)}\min_{t\in [t_0,t_0+\delta]}|w(t)|^p.
\end{align}
Note that \eqref{100.1} yields that
\begin{equation}\label{6.14.2}
w(0)E_\alpha(-at^\alpha)= \frac{w(0)}{a\Gamma(1-\alpha)}\frac{1}{t^\alpha}+O(\frac{1}{t^{2\alpha}}),\
w'(0)tE_{\alpha,2}(-at^\alpha)= \frac{w'(0)}{a\Gamma(2-\alpha)}\frac{1}{t^{\alpha-1}}+O(\frac{1}{t^{2\alpha-1}}),
\end{equation}
as $t\rightarrow +\infty$. Moreover $\alpha+\gamma>2$ implies $1-\gamma<\alpha-1$. Hence it follows from \eqref{6.14.1} and \eqref{6.14.2} that
$\liminf_{t\rightarrow +\infty}t^{1-\gamma} w(t)>0$.

(iv) Since $w'(0)=0$, it follows from \eqref{6.14.1} that
\begin{align*}
w(t)\geq E_\alpha(-at^\alpha)w(0)+ \frac{b\delta t^{\gamma-1}}{2a\Gamma(\gamma)}\min_{t\in [t_0,t_0+\delta]}|w(t)|^p.
\end{align*}
Thus we can obtain the desired conclusion by \eqref{6.14.2}.

(v) In this case, we have $1-\gamma\geq\alpha-1$. Thus, using \eqref{6.14.1}, \eqref{6.14.2} and the fact that $w'(0)>0$, we get $\liminf_{t\rightarrow +\infty}t^{\alpha-1} w(t)>0$.

(vi) Since $0<\gamma<\frac{\alpha}{2}$ and $1<\alpha<2$, we can choose $\beta$ satisfying $\frac{\alpha}{2}\leq \gamma+\beta<1$ and $0<\beta<\alpha$. From \eqref{8.11.1} and \eqref{7.25.1}, we know that
\begin{align}\label{7.26.2}
{}_0I_t^\beta w=&{}_0I_t^\beta E_\alpha(-at^\alpha)w(0)+{}_0I_t^\beta [tE_{\alpha,2}(-at^\alpha)w'(0)]+{}_0I_t^\beta
\int_0^t(t-s)^{\gamma+\alpha-1}E_{\alpha,\alpha+\gamma}(-a(t-s)^\alpha)f(s) ds\nonumber\\
=&t^\beta E_{\alpha,1+\beta}(-at^\alpha)w(0)+t^{1+\beta}E_{\alpha,2+\beta}(-at^\alpha)w'(0)\nonumber\\
&+\int_0^t (t-s)^{\gamma+\alpha+\beta-1}E_{\alpha,\alpha+\gamma+\beta}(-a(t-s)^\alpha)f(s) ds,
\end{align}
where we have used the fact that
\begin{align}\label{6.16.8}
&{}_0I_t^\beta
\int_0^t(t-\tau)^{\gamma+\alpha-1}E_{\alpha,\alpha+\gamma}(-a(t-\tau)^\alpha)f(\tau) d\tau\nonumber\\
&=\frac{1}{\Gamma(\beta)}\int_0^t\int_0^{s}(t-s)^{\beta-1}(s-\tau)^{\alpha+\gamma-1}
E_{\alpha,\alpha+\gamma}(-a(s-\tau)^\alpha)f(\tau)d\tau ds\nonumber\\
&=\frac{1}{\Gamma(\beta)}\int_0^t\int_\tau^{t}(t-s)^{\beta-1}(s-\tau)^{\alpha+\gamma-1}
E_{\alpha,\alpha+\gamma}(-a(s-\tau)^\alpha)f(\tau) dsd\tau\nonumber\\
&=\frac{1}{\Gamma(\beta)}\int_0^t\int_0^{t-\tau}(t-\tau-h)^{\beta-1}h^{\alpha+\gamma-1}
E_{\alpha,\alpha+\gamma}(-ah^\alpha) dhf(\tau)d\tau\nonumber\\
&=\int_0^t(t-\tau)^{\alpha+\gamma+\beta-1}
E_{\alpha,\alpha+\gamma+\beta}(-a(t-\tau)^\alpha)f(\tau)d\tau.
\end{align}
Furthermore, since $\gamma+\beta\geq\frac{\alpha}{2}$, we know $E_{\alpha,\alpha+\gamma+\beta}(-z)>0$ for every $z\in \mathbb{R}$ (see Theorem 2 in \cite{Pskhu} ). Thus, using \eqref{7.26.2}, and by an argument similar to the proof of Property(iii), we have that for $t$ large enough,
\begin{align}\label{6.14.3}
{}_0I_t^\beta w
\geq& t^\beta E_{\alpha,1+\beta}(-at^\alpha)w(0)+t^{1+\beta}E_{\alpha,2+\beta}(-at^\alpha)w'(0)\nonumber\\
&+\int_0^t(t-\tau)^{\alpha+\gamma+\beta-1}
E_{\alpha,\alpha+\gamma+\beta}(-a(t-\tau)^\alpha)|w(\tau)|^pd\tau\nonumber\\
\geq & t^\beta E_{\alpha,1+\beta}(-at^\alpha)w(0)+t^{1+\beta}E_{\alpha,2+\beta}(-at^\alpha)w'(0)+\frac{\delta t^{\gamma+\beta-1}}{2a\Gamma(\gamma+\beta)}\min_{t\in [t_0,t_0+\delta]}|w(t)|^p.
\end{align}
Observe that $\beta <\alpha$ and
\begin{equation}\label{6.14.4}
w(0)t^\beta E_{\alpha, 1+\beta}(-at^\alpha)= \frac{w(0)}{a\Gamma(1+\beta-\alpha)}\frac{1}{t^{\alpha-\beta}}+O(\frac{1}{t^{2\alpha-\beta}}),
\end{equation}
\begin{equation}\label{6.14.5}
w'(0)t^{1+\beta}E_{\alpha,2+\beta}(-at^\alpha)= \frac{w'(0)}{a\Gamma(2+\beta-\alpha)}\frac{1}{t^{\alpha-\beta-1}}+O(\frac{1}{t^{2\alpha-\beta-1}}),
\end{equation}
as $t\rightarrow +\infty$. Then it follows from \eqref{6.14.3}, \eqref{6.14.4} and \eqref{6.14.5}  that $\liminf_{t\rightarrow +\infty}t^{1-\gamma-\beta} {}_0I_t^\beta w>0$ if $\alpha+\gamma>2$, or $\alpha+\gamma\leq 2$ and $w'(0)=0$.

(vii) The assumption that $\alpha+\gamma\leq 2$  yields $\alpha-\beta-1\leq 1-\gamma-\beta$. Thus the desired conclusion can be obtained by \eqref{6.14.3}, \eqref{6.14.4} and \eqref{6.14.5}.

(viii) Suppose that $T=+\infty.$  The proof is divided into four cases.

\textbf{Case 1.} Assume that condition (a) is satisfied. We deduce from \eqref{6.16.1} that
\begin{equation}\label{6.16.2}
\int_0^T|w|^p \psi_T dt\leq K_1T^{1-\frac{p\gamma}{p-1}}+K_2T^{1-\frac{p(\alpha+\gamma)}{p-1}}
+|w(0)|T^{1-\alpha-\gamma}+|w'(0)|T^{2-\alpha-\gamma}.
\end{equation}
Thus, if $p(1-\gamma)<1$, then taking $T\rightarrow +\infty$, we get $\int_0^{+\infty}|w|^p  dt=0 $ by \eqref{6.16.2}. This implies $w\equiv0$, which contradicts the assumption that $w\not\equiv0$.

If $p(1-\gamma)=1$, then using \eqref{6.16.2} and taking $T\rightarrow +\infty$, we obtain that $\int_0^{+\infty}|w|^p  dt<+\infty.$ Additionally, if $\gamma\geq\frac{\alpha}{2}$, Property (iii) implies that there exist constants $C>0$ and $L>0$ such that $w(t)\geq Ct^{\gamma-1}$ for $t\geq L$, which yields that
\[C\int_L^{+\infty}t^{-1}dt=C\int_L^{+\infty}t^{-p(1-\gamma)}\leq \int_0^{+\infty}|w|^p  dt.\]
This contradicts $\int_0^{+\infty}|w|^p  dt<+\infty.$ On the other hand,  if $0<\gamma<\frac{\alpha}{2}$, Property (vi) implies that there exist constants $C>0$ and $\tilde{L}>0$ such that ${}_0I_t^\beta w(t)\geq Ct^{\gamma+\beta-1}$ for $t\geq \tilde{L}$, where $\beta$ satisfies $\frac{\alpha}{2}\leq \gamma+\beta<1$ and $\beta<\alpha$.
By the weighted estimate of the operator ${}_0I_t^\beta$, the inequality
\[
\Big(\int_0^{+\infty}t^{-\beta p} |{}_0I_t^\beta w|^p dt\Big)^{\frac{1}{p}} \leq \frac{\Gamma(1-1/ p)}{\Gamma(\beta+1-1/p)}\Big(\int_0^{+\infty}|w(t)|^p  dt\Big)^{\frac{1}{p}}
\]
is valid (see Lemma 2.3 (a) in \cite{kilbas2006} or Exercise 28 in Chapter 6 of \cite{folland}). As a result,
\[
C\int_{\tilde{L}}^{+\infty}t^{-1} dt= C\int_{\tilde{L}}^{+\infty}t^{-\beta p-p(1-\gamma)+p\beta} dt\leq\Big(\frac{\Gamma(1-1/ p)}{\Gamma(\beta+1-1/p)}\Big)^p\int_0^{+\infty}|w(t)|^p  dt,
\]
which again contradicts $\int_0^{+\infty}|w|^p  dt<+\infty.$

\textbf{Case 2.} Assume that condition (b) is satisfied, that is $\alpha+\gamma\leq 2$, $w'(0)=0$ and $p(1-\gamma)\leq 1$. Then it follows from \eqref{6.16.1} that
\[
\int_0^T|w|^p \psi_T dt\leq K_1T^{1-\frac{p\gamma}{p-1}}+K_2T^{1-\frac{p(\alpha+\gamma)}{p-1}}
+|w(0)|T^{1-\alpha-\gamma}.
\]
If $p(1-\gamma)<1$, then we have $w\equiv0$ by taking $T\rightarrow+\infty$. This contradicts the assumption that $w\not\equiv0$. If $p(1-\gamma)=1$, then using Property(iv)(vi) and repeating the arguments in case 1, we can also obtain a contradiction.

\textbf{Case 3.} Assume that condition (d) is satisfied. Then from \eqref{6.16.1} we have
\[
w'(0)T^{2-\alpha-\gamma}\leq K_1T^{1-\frac{p\gamma}{p-1}}+K_2T^{1-\frac{p(\alpha+\gamma)}{p-1}}
+|w(0)|T^{1-\alpha-\gamma}.
\]
In other words,
\begin{equation}\label{6.16.5}
w'(0)\leq K_1T^{-1-\frac{\gamma}{p-1}+\alpha}+K_2T^{-1-\frac{\alpha+\gamma}{p-1}}
+|w(0)|T^{-1}.
\end{equation}
Observe that $\alpha-1<\frac{\gamma}{p-1}$. Thus taking $T\rightarrow +\infty$, we deduce from \eqref{6.16.5} that $w'(0)=0$, which contradicts the assumption that $w'(0)>0$.

\textbf{Case 4.} Assume that condition (c) holds. For the case $p(1-\gamma)<1$, the assumption $\alpha+\gamma=2$ yields $\alpha-1<\frac{\gamma}{p-1}$. Thus $T<+\infty$ by case 3. For the case $p(1-\gamma)=1$, we can also obtain a contradiction by repeating the arguments in case 1 and using Property(v), (vii). Hence $T<+\infty$.
\end{proof}
\begin{remark}
In \cite{T.Cazenave}, for problem \eqref{1.7-1} with $\alpha=1$, some similar results were proved by comparison technique and shifting the time. The authors of \cite{zhangli3} overcame the technical difficulty, given by the memory effect of the equation, and extended the results to the time fractional case ($0<\alpha<1$). Their proof depends on the nonnegativity of the Mittag-Leffler function $E_{\alpha,\alpha}(t)$ for $t<0$. But, in the case $1<\alpha<2$, $E_{\alpha,\alpha}(t)$ is not nonnegative on $(-\infty,0)$. Hence, the method using in \cite{T.Cazenave,zhangli3} could not be applied to study problem \eqref{1.7-1}. Here, we consider the asymptotic behavior of $_0I_t^\beta w$ for some $\beta>0$ instead of $w$ in some cases. Accordingly, we can prove Property(viii), which is crucial to prove our blow-up results.
\end{remark}
\begin{corollary}\label{coro1}
Let $T>0$, $\gamma>0$, $1<\alpha<2$,  $p>1$, $a,b>0$, and  $\beta\in [\frac{\alpha}{2},\alpha)$. If $w\in C^1([0,T])$, and $w\not \equiv0$ satisfies that $_0I_t^{2-\alpha}(w-w'(0)t-w(0))\in AC^2([0,T])$ and for almost every $t\in [0,T]$,
\begin{equation}\label{6.16.2-1}
_0D_t^\alpha w+aw\geq b\cdot {}_0^{}I_t^{\gamma}|w|^p,
\end{equation}
then the following properties hold.
\begin{enumerate}
  \item[\rm (i)] There exist positive constants $K_1$ and $K_2$ independent of $T$ such that \[Tw'(0)+w(0)\leq K_1T^{\alpha+\gamma-\frac{p\gamma}{p-1}}
      +K_2T^{-\frac{\alpha+\gamma}{p-1}}.\]
\item[\rm (ii)] If $T=+\infty,$  then $\liminf_{t\rightarrow +\infty}|w(t)|=0$ and $\liminf_{t\rightarrow +\infty}t^{\min\{\frac{\alpha+\gamma-1}{p},\frac{\gamma}{p-1}\}}|w(t)|<+\infty$.

\item[\rm (iii)]If $T=+\infty,$ $w'(0)=0$, then $\liminf_{t\rightarrow +\infty}t^{1-\gamma-\beta} {}_0I_t^\beta w>0$.

\item[\rm(iv)] If $T=+\infty,$ $\alpha+\gamma\leq 2$ and $w'(0)>0$, then $\liminf_{t\rightarrow +\infty}t^{\alpha-\beta-1} {}_0I_t^\beta w>0$.

\item[\rm(v)] If $T=+\infty,$ $\alpha+\gamma>2$, then $\liminf_{t\rightarrow +\infty}t^{1-\gamma-\beta} {}_0I_t^\beta w>0$.

\item[\rm(vi)] If one of the following conditions is
satisfied:
\begin{enumerate}
\item[\rm(a)] $\alpha+\gamma> 2$ and $p(1-\gamma)\leq 1$;
\item[\rm(b)] $\alpha+\gamma\leq 2$, $w'(0)=0$ and $p(1-\gamma)\leq 1$;
\item[\rm(c)] $\alpha+\gamma=2$, $w'(0)>0$ and $p(1-\gamma)\leq 1$;
\item[\rm(d)] $\alpha+\gamma<2$, $w'(0)>0$ and $p<1+\frac{\gamma}{\alpha-1}$,
\end{enumerate}
then $T<+\infty$.
\end{enumerate}
\end{corollary}
\begin{proof}
According to the proof of Lemma \ref{lemma17}(i)(ii), we know that Property (i) and (ii) also hold in this case.

Denote $_0D_t^\alpha w+aw=g(t)$. Then
\begin{align}\label{6.16.7}
w(t)= E_\alpha(-at^\alpha)w(0)+tE_{\alpha,2}(-at^\alpha)w'(0)+
\int_0^t(t-s)^{\alpha-1}E_{\alpha,\alpha}(-a(t-s)^\alpha)g(s) ds.
\end{align}
Take $\beta\in [\frac{\alpha}{2},\alpha)$. It follows from \eqref{6.16.7}, \eqref{7.25.1} and \eqref{6.16.8} that
\begin{align}\label{6.17.1}
{}_0I_t^\beta w=t^\beta E_{\alpha,1+\beta}(-at^\alpha)w(0)+t^{1+\beta}E_{\alpha,2+\beta}(-at^\alpha)w'(0)+
\int_0^t (t-s)^{\alpha+\beta-1}E_{\alpha,\alpha+\beta}(-a(t-s)^\alpha)g(s) ds.
\end{align}
Noting that $2+\beta>\alpha+\beta\geq \frac{3\alpha}{2}$, we have $E_{\alpha,2+\beta}(-z)>0$ and $E_{\alpha,\alpha+\beta}(-z)>0$ for every $z\in \mathbb{R}$ ( see Theorem 2 in \cite{Pskhu}). Hence we conclude from \eqref{6.16.2-1}, \eqref{6.17.1} and \eqref{8.11.1} that
\begin{align*}
{}_0I_t^\beta w\geq &t^\beta E_{\alpha,1+\beta}(-at^\alpha)w(0)+t^{1+\beta}E_{\alpha,2+\beta}(-at^\alpha)w'(0)\\
&+b\int_0^t (t-s)^{\alpha+\beta-1}E_{\alpha,\alpha+\beta}(-a(t-s)^\alpha){}_0I_s^\gamma |w|^p ds\\
=& t^\beta E_{\alpha,1+\beta}(-at^\alpha)w(0)+t^{1+\beta}E_{\alpha,2+\beta}(-at^\alpha)w'(0)\\
&+b\int_0^t (t-s)^{\alpha+\beta+\gamma-1}E_{\alpha,\alpha+\beta+\gamma}(-a(t-s)^\alpha)|w(s)|^p ds.
\end{align*}
The rest of proof is similar to that for Lemma \ref{lemma17} (vi)-(viii), so
we omit it.
\end{proof}

When $\alpha=2$, we have the following results.
\begin{corollary}\label{coro2}
Let $T>0$, $\gamma>0$,  $p>1$, $a,b>0$, and  $\beta\geq 2$. If $w\not\equiv0$ satisfies that for almost every $t\in [0,T]$,
\begin{equation}\label{6.16.2-2}
w''+aw\geq b\cdot {}_0^{}I_t^{\gamma}|w|^p,
\end{equation}
then the following properties hold.
\begin{enumerate}
\item[\rm (i)] There exist positive constants $K_1$ and $K_2$ independent of $T$ such that \[Tw'(0)+w(0)\leq K_1T^{2+\gamma-\frac{p\gamma}{p-1}}
      +K_2T^{-\frac{2+\gamma}{p-1}}.\]
\item[\rm (ii)] If $T=+\infty,$  then $\liminf_{t\rightarrow +\infty}|w(t)|=0$ and $\liminf_{t\rightarrow +\infty}t^{\min\{\frac{1+\gamma}{p},\frac{\gamma}{p-1}\}}|w(t)|<+\infty$.

\item[\rm(iii)] If $T=+\infty,$ then $\liminf_{t\rightarrow +\infty}t^{1-\gamma-\beta} {}_0I_t^\beta w>0$.

\item[\rm(iv)] If $p(1-\gamma)\leq 1$, then $T<+\infty$.
\end{enumerate}
\end{corollary}
\begin{proof}
By the proof of Lemma \ref{lemma17}, we know that Property (i) and (ii) also hold in the case $\alpha=2$.

(iii) Denote $w''+aw=g(t)$. Then
\begin{align*}
w(t)= E_2(-at^2)w(0)+tE_{2,2}(-at^2)w'(0)+
\int_0^t(t-s)E_{2,2}(-a(t-s)^2)g(s) ds.
\end{align*}
Since $2+\beta>1+\beta\geq 3$, we know $E_{2,1+\beta}(-z)\geq0$, $E_{2,2+\beta}(-z)>0$ and $E_{2,2+\beta}(-z)>0$ for every $z\in \mathbb{R}$ ( see Theorem 2 in \cite{Pskhu}). Then it follows from \eqref{6.17.1} and \eqref{6.16.2-2} that
\begin{align*}
{}_0I_t^\beta w\geq& t^\beta E_{2,1+\beta}(-at^2)w(0)+t^{1+\beta}E_{2,2+\beta}(-at^2)w'(0)+ b\int_0^t (t-s)^{1+\beta}E_{2,2+\beta}(-a(t-s)^2){}_0I_s^\gamma |w|^p ds\\
=& t^\beta E_{2,1+\beta}(-at^2)w(0)+t^{1+\beta}E_{2,2+\beta}(-at^2)w'(0)\\
&+ b\int_0^t (t-s)^{1+\beta+\gamma}E_{2,2+\beta+\gamma}(-a(t-s)^2)|w(s)|^pds.
\end{align*}
Since $w\not\equiv0$, there exist $t_0\in [0,T)$ and $\delta>0$  such that  $w(t)\not=0$ for $t\in [t_0,t_0+\delta]$. Observe that $\beta-1<\beta+\gamma-1$, and  \eqref{100.2} yields that
\begin{equation*}
t^{2}E_{2,2+\beta+\gamma}(-at^2)= a^{-\frac{1+\beta+\gamma}{2}}t^{1-\beta-\gamma}\cos(t-\frac{(1+\beta+\gamma)\pi}{2}) +\frac{1}{a\Gamma(\beta+\gamma)}+O(\frac{1}{t^{2}}),
\end{equation*}
as $t\rightarrow +\infty$.
Then by an argument similar to the proof of estimate \eqref{6.14.1}, we have
\[{}_0I_t^\beta w\geq t^\beta E_{2,1+\beta}(-at^2)w(0)+t^{1+\beta}E_{2,2+\beta}(-at^2)w'(0)+\frac{b\delta t^{\beta+\gamma-1}}{2a\Gamma(\beta+\gamma)}\min_{t\in [t_0,t_0+\delta]}|w(t)|^p.\]
Furthermore, invoking $\beta-1<\beta+\gamma-1$ and noting that \eqref{100.2} implies that
\begin{equation*}
t^\beta E_{2, 1+\beta}(-at^2)= a^{-\frac{\beta}{2}}\cos(t-\frac{\beta\pi}{2})+ \frac{1}{a\Gamma(\beta-1)}t^{\beta-2}+O(\frac{1}{t^{4-\beta}}),
\end{equation*}
\begin{equation*}
t^{1+\beta}E_{2,2+\beta}(-at^2)= a^{-\frac{1+\beta}{2}}\cos(t-\frac{(1+\beta)\pi}{2}) +\frac{1}{a\Gamma(\beta)}t^{\beta-1}+O(\frac{1}{t^{3-\beta}}),
\end{equation*}
as $t\rightarrow +\infty$, we can obtain ${}_0I_t^\beta w\geq Ct^{\beta+\gamma-1}$ for $t$ large enough. This completes the proof.

(iv) The proof is similar to that of Lemma \ref{lemma17}(viii), so we omit it.
\end{proof}

\section{Finite time blow-up and Global existence}
\noindent

In this section, we give the local existence result of problem \eqref{20.1}, and prove sharp results on the blow-up and global existence of solution for  problem \eqref{20.1}.

Let $\lambda_1>0$ be the first eigenvalue of $-\triangle$ in $H_0^1(\Omega)$. We denote by $\varphi_1$ the corresponding positive eigenfunction with $\int_\Omega\varphi_1(x)dx=1$. Denote $-A=\triangle$. Let $1<q\leq +\infty$, and consider the operator $A$ defined on
\[D(A)=\left\{
         \begin{array}{ll}
           \{u\in \bigcap_{r\geq1}W^{2,r}(\Omega)\ | \ u,Au\in L^\infty(\Omega), u=0 \text{ on }\partial\Omega \}, & \hbox{if $q=\infty$,} \\
           \{u\in W_0^{1,q}(\Omega)\ |\ Au\in L^q(\Omega)\}, & \hbox{if  $1<q<+\infty$.}
         \end{array}
       \right.
\]

First, we give the definitions of the solution operators, which are similar to those in \cite{Zhangli4}.
\begin{definition}\label{def2}
Let $\alpha\in (1,2)$, $1<q\leq +\infty$.
For every $u_0\in L^q(\Omega)$, we define the operators $P_\alpha(t)$ and $S_\alpha(t)$ by
\begin{equation}\label{60.1}
P_\alpha(t)u_0=
\frac{1}{2\pi i}\int_\Gamma E_\alpha(\lambda t^\alpha)(\lambda I+A)^{-1}u_0d\lambda,\ t>0,\text{ and } P_\alpha(0)u_0=u_0,
\end{equation}
\begin{equation}\label{60.2}
S_\alpha(t)u_0=\frac{1}{2\pi i}\int_\Gamma E_{\alpha,\alpha}(\lambda t^\alpha)(\lambda I+A)^{-1}u_0d\lambda,\ t>0, \text{ and } S_\alpha(0)u_0=\frac{u_0}{\Gamma(\alpha)},
\end{equation}
where $\Gamma\in\{\gamma(\varepsilon,\varphi)\subseteq \rho(-A)\ |\ \varepsilon>0,\  \varphi \text{ satisfies }0<\varphi<\pi,\ \frac{\pi\alpha}{2}<\arg(-\lambda_1+\varepsilon e^{i\varphi})<\pi\}$. Here $\gamma(\varepsilon,\varphi)=\{re^{i\arg(-\lambda_1+\varepsilon e^{i\varphi})}|r\geq |-\lambda_1+\varepsilon e^{i\varphi}|\}\cup\{re^{-i\arg(-\lambda_1+\varepsilon e^{i\varphi})}|r\geq |-\lambda_1+\varepsilon e^{-i\varphi}|\}\cup\{-\lambda_1+\varepsilon e^{i\theta}|-\varphi\leq\theta\leq\varphi\}$.
\end{definition}
\begin{remark}
According to \eqref{100.1} and Cauchy's integral theorem, $P_\alpha(t)$ and $S_\alpha(t)$ are independent of $\varphi$ and $\varepsilon$, and then are well defined.
\end{remark}

By making use of the complex integral representations of the solution operators, we can derive some estimates of the operators $P_\alpha(t)$ and $S_\alpha(t)$.
\begin{lemma}\label{lemma100}
The operators $P_\alpha(t)$ and $S_\alpha(t)$ have the following properties.
\begin{enumerate}
\item[\rm{(i)}] For $u_0\in L^\infty(\Omega)$,  we have $P_\alpha(t)u_0\in C((0,+\infty), L^\infty(\Omega))\cap C([0,+\infty), L^q(\Omega))$ for every $q\in (1,+\infty)$, $\lim_{t\rightarrow 0^+}P_\alpha(t)u_0=u_0$ in the weak-star topology of $L^\infty(\Omega)$. Moreover, if $u_0\in L^s(\Omega) $ $(1<s\leq +\infty)$, then there exists a constant $C>0$ such that for $t\geq 0$,
\begin{equation}\label{100.7}
\|P_\alpha(t)u_0\|_{L^s(\Omega)}\leq \frac{C}{1+t^\alpha}\|u_0\|_{L^s(\Omega)},\ \ \ \|{}_0I_t^1P_\alpha(t)u_0\|_{L^s(\Omega)}\leq \frac{C}{1+t^{\alpha-1}}\|u_0\|_{L^s(\Omega)}.
\end{equation}
\item[\rm (ii)] For $u_0\in L^\infty(\Omega)$, $0<\gamma\leq 1$ and $t>0$, we have
$
S_\alpha(t)u_0=t^{1-\alpha}{}_0I_t^{\alpha-1}P_\alpha(t)u_0,
$
\[
{}_0I_t^\gamma [t^{\alpha-1}S_\alpha(t)u_0]={}_0I_t^{\alpha+\gamma-1}P_\alpha(t)u_0=\frac{t^{\alpha+\gamma-1}}{2\pi i}\int_\Gamma E_{\alpha,\alpha+\gamma}(\lambda t^\alpha)(\lambda I+A)^{-1}u_0d\lambda,
\]
$t^{\alpha-1}S_\alpha(t)u_0\in C([0,+\infty), L^\infty(\Omega))$ and $S_\alpha(t)u_0\in C([0,+\infty), L^q(\Omega))$ for every $q\in (1,+\infty)$.
Moreover, if $u_0\in L^s(\Omega) $ $(1<s\leq +\infty)$, then there exists a constant $C>0$ such that for $t\geq0$,
\begin{equation*}
\|S_\alpha(t)u_0\|_{L^s(\Omega)}\leq \frac{C}{1+t^{2\alpha}}\|u_0\|_{L^s(\Omega)},\ \
\|{}_0I_t^\gamma [t^{\alpha-1}S_\alpha(t)u_0]\|_{L^s(\Omega)}\leq \frac{C}{1+t^{1-\gamma}}\|u_0\|_{L^s(\Omega)}.
\end{equation*}
\item[{\rm (iii)}] If $u_0\in L^\infty(\Omega)$, then $P_\alpha(t)u_0\in C^1((0,+\infty),L^\infty(\Omega))$.
Moreover, $\int_\Omega (P_\alpha(t)u_0)vdx\in C^1([0,+\infty))$ for every $v\in D(A).$
\item[{\rm(iv)}] Let $T>0$, $q\geq1$ and $w=\int_0^t(t-s)^{\alpha-1}S_\alpha(t-s)f(s)ds$.
If $f\in L^q((0,T),L^r(\Omega))$ for some $r\in (1,+\infty)$
then $w\in C([0,T],L^r(\Omega))$. Furthermore, if $q(\alpha-1)>1,$ then $w\in C^{1,\alpha-1-\frac{1}{q}}([0,T],L^r(\Omega)).$
\end{enumerate}
\end{lemma}
\begin{proof}(i) To prove the first estimate in \eqref{100.7}, we take $\varepsilon_0>0$ small enough such that $\pi-\arcsin \frac{\varepsilon_0}{\lambda_1}>\frac{\pi\alpha}{2}$, and choose $\Gamma=\gamma(\varepsilon_0,\varphi)\in\rho(-A)$ with $\arg(-\lambda_1+\varepsilon_0 e^{i\varphi})\in(\frac{\pi\alpha}{2},\pi)$.
Denote $\varphi_0=\arg(-\lambda_1+\varepsilon_0 e^{i\varphi})$, $r_0=|-\lambda_1+\varepsilon_0 e^{i\varphi}|$  and $\Gamma=\gamma(\varepsilon_0,\varphi)=\Gamma_1\cup\Gamma_2\cup\Gamma_3,$  where $\Gamma_1=\{re^{i\varphi_0}\ |\ r\geq r_0\}$, $\Gamma_2=\{-\lambda_1+\varepsilon_0e^{i\theta}\ |\ -\varphi\leq\theta\leq\varphi\}$ and $\Gamma_3=\{re^{-i\varphi_0}\ |\ r\geq r_0\}.$
Note that $\frac{\pi\alpha}{2}<\pi-\arcsin\frac{\varepsilon_0}{\lambda_1}\leq|\arg z|\leq\pi$ for $z\in \Gamma$. Then, for $u_0\in L^s(\Omega)(1<s\leq +\infty)$ and $t>0$,  we deduce from \eqref{100.1} that there exists a constant $C>0$ such that
\begin{align}\label{7.28.1}
\bigg\|\frac{1}{2\pi i}\int_{\Gamma_k} E_{\alpha}(t^\alpha \lambda)(\lambda I+A)^{-1}u_0d\lambda\bigg\|_{L^s(\Omega)}
&\leq C \int_{\Gamma_k}\frac{1}{1+t^\alpha|\lambda|}
\frac{1}{|\lambda+\lambda_1|}|d\lambda|\|u_0\|_{L^s(\Omega)}\nonumber\\
&\leq C\int_{r_0}^{+\infty}\frac{dr\|u_0\|_{L^s(\Omega)}}{(1+t^\alpha r)\sqrt{(\lambda_1+r\cos\varphi_0)^2+r^2\sin^2\varphi_0}}\nonumber\\
&\leq C\int_{r_0}^{+\infty}\frac{1}{1+t^\alpha r}\frac{1}{r}dr\|u_0\|_{L^s(\Omega)}\nonumber\\
&\leq \frac{C}{t^\alpha}\|u_0\|_{L^s(\Omega)},\ k=1,3,
\end{align}
and
\begin{align}\label{7.28.2}
\bigg\|\frac{1}{2\pi i}\int_{\Gamma_2} E_{\alpha}(t^\alpha \lambda)(\lambda I+A)^{-1}u_0d\lambda\bigg\|_{L^s(\Omega)}
&\leq C \int_{\Gamma_2}\frac{1}{1+t^\alpha|\lambda|}\frac{1}{|\lambda+\lambda_1|}|d\lambda|
\|u_0\|_{L^s(\Omega)}\nonumber\\
&\leq C\int_{-\varphi}^\varphi \frac{1}{1+t^\alpha(\lambda_1-\varepsilon_0)}d\theta\|u_0\|_{L^s(\Omega)}\nonumber\\
&\leq \frac{C}{t^\alpha}\|u_0\|_{L^s(\Omega)}.
\end{align}

On the other hand, by Cauchy's integral theorem, we can take $\Gamma=\tilde{\Gamma}_1\cup\tilde{\Gamma}_2\cup\tilde{\Gamma}_3$, where $\tilde{\Gamma}_1=\{re^{i\tilde{\varphi}}\ |\ r\geq\frac{1}{t^\alpha}\}$, $\tilde{\Gamma}_2=\{\frac{1}{t^\alpha}e^{i\theta}\ |\ -\tilde{\varphi}\leq\omega\leq\tilde{\varphi}\}$ and $\tilde{\Gamma}_3=\{re^{-i\tilde{\varphi}}\ |\ r\geq\frac{1}{t^\alpha}\}, \tilde{\varphi}\in (\frac{\pi\alpha}{2},\pi).$ Then
there exists a constant $C>0$ such that
\begin{align}\label{7.28.3}
\bigg\|\frac{1}{2\pi i}\int_{\tilde{\Gamma}_k} E_{\alpha}(t^\alpha\lambda)(\lambda I+A)^{-1}u_0d\lambda\bigg\|_{L^s(\Omega)}
&\leq C\int_{\frac{1}{t^\alpha}}^{+\infty}\frac{\|u_0\|_{L^s(\Omega)}}{r(1+t^\alpha r)}dr
&\leq C\|u_0\|_{L^s(\Omega)},\ k=1,3,
\end{align}
and
\begin{align}\label{7.28.4}
\bigg\|\frac{1}{2\pi i}\int_{\tilde{\Gamma}_2} E_{\alpha}(t^\alpha \lambda)(\lambda I+A)^{-1}u_0d\lambda\bigg\|_{L^s(\Omega)}
\leq C \int_{-\tilde{\varphi}}^{\tilde{\varphi}}\frac{1}{|\lambda|}|d\lambda|\|u_0\|_{L^s(\Omega)}\leq C\|u_0\|_{L^s(\Omega)}.
\end{align}
Combining \eqref{7.28.1}-\eqref{7.28.4}, we get that $\|P_\alpha(t)u_0\|_{L^s(\Omega)}\leq \frac{C}{1+t^\alpha}\|u_0\|_{L^s(\Omega)}$ for some $C>0$.

Next, we prove the second estimate in \eqref{100.7}. By \eqref{100.1} and $4> \frac{3\alpha}{2}$, there exists a constant $C>0$ such that $E_{\alpha,4}(-t^\alpha)\geq \frac{C}{1+t^\alpha}$. Thus
\[
\|{}_0I_t^1P_\alpha(t)u_0\|_{L^s(\Omega)}\leq C{}_0I_t^1 E_{\alpha,4}(-t^\alpha)\|u_0\|_{L^s(\Omega)}
=CtE_{\alpha,5}(-t^\alpha)\|u_0\|_{L^s(\Omega)}\leq\frac{C}{1+t^{\alpha-1}}\|u_0\|_{L^s(\Omega)}.
\]

Using the dominated convergence theorem, we can find that $P_\alpha(t)u_0\in C((0,+\infty), L^\infty(\Omega))$. Moreover, a density argument and \eqref{100.7} show that $\lim_{t\rightarrow 0^+}P_\alpha(t)u_0=u_0$ in  $L^q(\Omega)$. Thus $u\in C([0,+\infty), L^q(\Omega))$.

Finally, we prove $\lim_{t\rightarrow 0^+}P_\alpha(t)u_0=u_0$ in the weak-star topology of $L^\infty(\Omega)$. In fact, by the definition of the operator $P_\alpha(t)$ and taking $\varepsilon>\lambda_1$, we have
\begin{align*}
P_\alpha(t)u_0-u_0=\frac{1}{2\pi i}\int_\Gamma E_{\alpha}(t^\alpha \lambda)(\lambda I+A)^{-1}u_0d\lambda-u_0
=-\frac{1}{2\pi i}\int_\Gamma E_{\alpha}(t^\alpha\lambda)\frac{1}{\lambda}A(\lambda I+A)^{-1}u_0d\lambda.
\end{align*}
This and the dominated convergence theorem  imply that for every $v\in C_0^\infty(\Omega)$,
\begin{align*}
\int_\Omega [P_\alpha(t)u_0-u_0]vdx
=-\int_\Omega\frac{1}{2\pi i}\int_\Gamma E_{\alpha}(t^\alpha\lambda)\frac{1}{\lambda}(\lambda I+A)^{-1}u_0d\lambda Avdx\rightarrow 0,
\end{align*}
as $t\rightarrow 0^+$. Thus we can complete the proof by \eqref{100.7} and the fact that $C_0^\infty(\Omega)$ is dense
in $L^1(\Omega)$.

(ii) The proof is similar to that of (i), so we omit it.

(iii) Using the dominated convergence theorem, we find
\begin{align}\label{7.28.6}
P_\alpha'(t)u_0=\frac{t^{\alpha-1}}{2\pi i}\int_\Gamma \lambda E_{\alpha,\alpha}(t^\alpha \lambda)(\lambda I+A)^{-1}u_0 d\lambda=-\frac{t^{\alpha-1}}{2\pi i}\int_\Gamma E_{\alpha,\alpha}(t^\alpha \lambda)A(\lambda I+A)^{-1}u_0 d\lambda.
\end{align}
Then $P_\alpha(t)u_0\in C^1((0,+\infty),L^\infty(\Omega))$ by the dominated convergence theorem. Furthermore, for $v\in D(A)$, we deduce from \eqref{7.28.6} that
\begin{align*}
\int_\Omega [P_\alpha'(t)u_0]vdx=-\frac{t^{\alpha-1}}{2\pi i}\int_\Omega\int_\Gamma E_{\alpha,\alpha}(t^\alpha \lambda)(\lambda I+A)^{-1}u_0 d\lambda Avdx\rightarrow 0,
\end{align*}
as $t\rightarrow 0^+$. Note that $\int_\Omega [P_\alpha(t)u_0]vdx\in C([0,+\infty))$ by (i). Consequently, $\int_\Omega (P_\alpha(t)u_0)vdx\in C^1([0,+\infty))$.

(iv)The proof is similar to that of Lemma 3.5 in \cite{Zhangli4}, so we omit it.
\end{proof}

According to the results in \cite{Zhangli4}, we can give the definition of the mild solution of \eqref{20.1}
\begin{definition}
Let $T>0$ and $u_0,u_1\in L^\infty(\Omega)$. A function  $u\in C((0,T],L^\infty(\Omega))$ is said to be a mild solution of problem
\eqref{20.1} if $u$ satisfies
 $\lim_{t\rightarrow 0^+}\|u(t)-P_\alpha(t)u_0\|_{L^\infty(\Omega)}=0$ and
\[
u=P_\alpha(t)u_0+{}_0I_t^1P_\alpha(t)u_1+
\int_0^t(t-s)^{\alpha-1}S_\alpha(t-s){_0I_s^{\gamma}}(|u|^{p})ds.
\]
\end{definition}

For problem \eqref{20.1}, we have the following local existence result.
\begin{theorem}\label{theorem1}
Let $1<\alpha<2$, $\gamma> 0$ and $p>1$. For given $u_0,u_1\in L^\infty(\Omega)$, there exists $T=T(u_0,u_1)>0$ such that problem \eqref{20.1} admits a unique mild solution $u$ in
$C((0,T],L^\infty(\Omega))\cap L^\infty((0,T),L^\infty(\Omega))\cap C([0,T],L^q(\Omega))$ for every $q\in (1,\infty)$. The solution $u$ can be extended to a maximal interval $[0,T^{*})$, and either $T^{*}=+\infty$ or
$T^{*}<+\infty$ and
$\limsup_{t\rightarrow T^{*-}}\|u(t)\|_{L^\infty(\Omega)}=+\infty$.
\end{theorem}
\begin{proof}
For given $T>0$, let
$
E_{T}=L^\infty((0,T),L^\infty(\Omega)).
$
Let $B_K$ denote the closed ball in $E_{T}$ with center $0$ and radius $K$.
For given $u_0,u_1\in L^\infty(\Omega)$, we define the operator $G$ on $E_{T}$ as
\begin{equation}\label{7.27.1}
G(u)(t)=P_\alpha(t)u_0+{}_0I_t^1P_\alpha(t)u_1+
\int_0^t(t-s)^{\alpha-1}S_\alpha(t-s){_0I_s^{\gamma}}(|u|^{p})ds.
\end{equation}
Choose $M\geq \|u_0\|_{L^{\infty}(\Omega)}+T\|u_1\|_{L^{\infty}(\Omega)}$. Then
it follows from Lemma \ref{lemma100} that for $u\in B_K$, $t\in(0,T)$ and some constant $C>0$,
\begin{align}\label{7.27.2}
\|G(u)(t)\|_{L^{\infty}(\Omega)}
&\leq C [\|u_0\|_{L^\infty(\Omega)}+T\|u_1\|_{L^\infty(\Omega)}]
+C\int_0^t(t-\tau)^{\alpha-1}\tau^{\gamma}
\|u(\tau)\|_{L^{\infty}(\Omega)}^pd\tau\nonumber\\
&\leq C(M+T^{\alpha+\gamma}K^p),
\end{align}
\begin{equation}\label{7.27.3}
\|G(u)-G(v)\|_{L^{\infty}(\Omega)}\leq C K^{p-1}T^{\alpha+\gamma}\|u(t)-v(t)\|_{L^\infty((0,T),L^\infty(\Omega))}.
\end{equation}
From \eqref{7.27.2} and \eqref{7.27.3}, we can choose $T$ and $M$ such that $G$ is a strict contractive mapping on $B_K$. Thus $G$ possesses a unique fixed point $u \in B_K$. Note that $\int_0^t(t-s)^{\alpha-1}S_\alpha(t-s){_0I_s^{\gamma}}(|u|^{p})ds\in C([0,T],L^{\infty}(\Omega))$ by the dominated convergence theorem. Furthermore, using Lemma \ref{lemma100},  we know that $u\in C((0,T],L^{\infty}(\Omega))\cap C([0,T],L^{q}(\Omega))$ for every $q\in (1,+\infty)$, and $\lim_{t\rightarrow 0^+}\|u(t)-P_\alpha(t)u_0\|_{L^\infty(\Omega)}=0$. The uniqueness of the mild solution follows from Gronwall's inequality.

Set
\[
T^*=\sup\{T\ |\ u\in L^\infty((0,T),L^{\infty}(\Omega))\cap C([0,T],L^{q}(\Omega)) \text{ is a mild solution of \eqref{20.1}}\}.
\]
By an analogous argument to that of Theorem 4.5 in \cite{Zhangli4},  we can prove that if $T^* <+\infty$ and $\|u\|_{L^\infty((0,T^*),L^\infty(\Omega))}<+\infty$, then $\lim_{t\rightarrow T^{*-}}u(t)$ exists in $L^{\infty}(\Omega)$. This implies that $u$ can be extended after $T^*$, which contradicts the definition of $T^*$. Thus we obtain the desired conclusion.
\end{proof}

Recalling the formula of integration by parts, we can give the following definition of weak solution of \eqref{20.1}. Moreover, we also obtain the relationship between weak solutions and mild solutions of \eqref{20.1}.
\begin{definition}\label{def1}
Let $1<\alpha\leq 2$, $u_0,u_1\in L^1(\Omega)$ and $T>0$. We say that $u\in L^p((0,T),L^p(\Omega))$ is a weak solution of \eqref{20.1} if
\[
\int_{\Omega}\int_0^T[{_0I_t^{\gamma}}(|u|^{p})\varphi+(u_0+tu_1) ({_tD_T^\alpha}\varphi)] dtdx
=\int_{\Omega}\int_0^Tu(-\triangle\varphi)dtdx
+\int_{\Omega}\int_0^Tu({_tD_T^\alpha}\varphi) dtdx
\]
for every $\varphi\in C^{2,2}([0,T]\times \bar{\Omega})$ with $\varphi=0$ on $ \partial\Omega$ and $\varphi(T,x)=\varphi_t(T,x)=0$ for $x\in \bar{\Omega}$. Moreover, we call $u$ a global weak solution of \eqref{20.1} if $T > 0$ can be arbitrarily
chosen.
\end{definition}

\begin{lemma}\label{lemma6.28}
Let $T>0$, $u_0,u_1\in L^\infty(\Omega)$. If $u\in C((0,T],L^{\infty}(\Omega))$ is a mild solution of \eqref{20.1} obtained by Theorem \ref{theorem1}, then $u$ is also a weak solution of \eqref{20.1}.
\end{lemma}
\begin{proof}
By Theorem \ref{theorem1}, we know $u\in C([0,T],L^{q}(\Omega))$ for every $q\in (1,+\infty)$. Thus an argument similar to the proof of Lemma 5.2 in \cite{Zhangli4} shows that $u$ is also a weak solution of \eqref{20.1}.
\end{proof}

Then we give blow-up results of problem \eqref{20.1} with $1<\alpha<2$.
\begin{theorem}\label{th1}
Let $p>1$, $\gamma>0$, $1<\alpha<2$ and $u_0,u_1\in L^\infty(\Omega)$. If one of the following conditions is
satisfied:
\begin{enumerate}
\item[\rm(a)] $\alpha+\gamma> 2$ and $p(1-\gamma)\leq 1$;
\item[\rm(b)] $\alpha+\gamma\leq 2$, $\int_{\Omega}u_1(x)\varphi_1(x)dx=0$ and $p(1-\gamma)\leq 1$;
\item[\rm(c)] $\alpha+\gamma=2$, $\int_{\Omega}u_1(x)\varphi_1(x)dx>0$ and $p(1-\gamma)\leq 1$;
\item[\rm(d)] $\alpha+\gamma<2$, $\int_{\Omega}u_1(x)\varphi_1(x)dx>0$ and $p<1+\frac{\gamma}{\alpha-1}$,
\end{enumerate}
then all nonzero mild solutions of \eqref{20.1} does not exist globally in time.
\end{theorem}
\begin{proof}
By the regularity theory of elliptic equations\cite{Gilbarg}, we know the eigenfunction $\varphi_1\in C^2(\bar{\Omega})$ and $\varphi_1(x)=0$ on $\Omega$.
Suppose that $u$ is a global mild solution of \eqref{20.1}. Then $u\in C([0,+\infty),L^{q}(\Omega))$ for every $q\in (1,+\infty)$, and $u$ is also a global weak solution of \eqref{20.1} by Lemma \ref{lemma6.28}. For  $T>0$, we take $\varphi(t,x)=\psi_T(t)\varphi_1(x)$ as a test function, where
$\psi_T\in C^2([0,T])$ satisfies $\psi_T(T)=\psi_T'(T)=0$, and then
\begin{align}\label{6.24.1}
\int_{\Omega}\int_0^T [{_0I_t^{\gamma}}(|u|^{p})\varphi_1\psi_T+
(u_0+tu_1)\varphi_1({_tD_T^\alpha}\psi_T)]dtdx
=\int_{\Omega}\int_0^T[\lambda_1u\varphi_1
\psi_T+u
\varphi_1({_tD_T^\alpha}\psi_T)]dtdx.
\end{align}
Denote $w(t)=\int_{\Omega}u\varphi_1 dx$.  From Lemma \ref{lemma100}, we deduce $w\in C^1([0,T])$.  Then \eqref{6.24.1} and \eqref{112-4} yield that
\begin{align}\label{6.24.2}
&\int_0^T{_0I_t^{2-\alpha}}[w-w(0)-tw'(0)]{\psi_T}'' dt+\lambda_1\int_0^Tw\psi_T dt\nonumber\\
&=\int_0^T[w-w(0)-tw'(0)]{_tD_T^{\alpha}}\psi_T dt+\lambda_1\int_0^Tf\psi_T dt\nonumber\\
&=\int_0^T\int_{\Omega} [{_0I_t^{\gamma}}(|u|^{p})\varphi_1dx\psi_Tdt.
\end{align}
Due to the arbitrariness of $\psi_T$, we obtain
\begin{equation}\label{6.28.1}
\frac{d^2}{dt^2}{_0I_t^{2-\alpha}}[w-w(0)-tw'(0)]+\lambda_1w(t)=\int_{\Omega} {_0I_t^{\gamma}}(|u|^{p})\varphi_1dx
\end{equation}
in the sense of distributions. In addition, the fact that $u\in C([0,+\infty),L^{q}(\Omega))$ for every $q\in (1,+\infty)$ implies  $\int_{\Omega} {_0I_t^{\gamma}}(|u|^{p})\varphi_1dx \in C([0,T])$, and Lemma \ref{lemma100} yields ${_0I_t^{2-\alpha}}[w-w(0)-tw'(0)]\in C([0,T])$. Thus it follows from the regularity theory  that the equality \eqref{6.28.1} holds for $t\in [0,T]$ in the classical sense. In other words,
\[
{}_0D_t^\alpha w+\lambda_1w(t)=\int_{\Omega} {_0I_t^{\gamma}}(|u|^{p})\varphi_1dx,\ t\in [0,T].
\]
From Jensen's inequality, we have
\begin{equation}\label{6.28.2}
{}_0D_t^\alpha w+\lambda_1w(t)=\int_{\Omega} {_0I_t^{\gamma}}(|u|^{p})\varphi_1dx
\geq {_0I_t^{\gamma}}
\Big(\int_{\Omega}|u|\varphi_1dx\Big)^p
\geq {_0I_t^{\gamma}} |w|^p  ,\ \ t\in [0,T].
\end{equation}
Then \eqref{6.28.2} and Corollary \ref{coro1}(vi) yield a contradiction. This completes the
proof.
\end{proof}

Next, we give a blow-up result of the semilinear wave equation ( i.e. \eqref{20.1} with $\alpha=2$).
\begin{theorem}\label{th20}
Let $p>1$, $\gamma>0$, $\alpha=2$ and $u_0,u_1\in L^1(\Omega)$. Assume $u\in C([0,T],L^p(\Omega))$ is a weak solution of \eqref{20.1}.
If $p(1-\gamma)\leq 1$, then $T<+\infty$.
\end{theorem}
\begin{proof}
In this case, we  know that for every $\psi_T\in C^2([0,T])$ with $\psi_T(T)=\psi_T'(T)=0$,
\begin{align*}
\int_{\Omega}\int_0^T [{_0I_t^{\gamma}}(|u|^{p})\varphi_1\psi_T+
(u_0+tu_1)\varphi_1\psi_T'']dtdx
=\int_{\Omega}\int_0^T[\lambda_1u\varphi_1
\psi_T+u
\varphi_1\psi_T'']dtdx.
\end{align*}
Denote $w(t)=\int_{\Omega}u\varphi_1 dx$. Since $u\in C([0,T],L^p(\Omega))$, we have $w\in C([0,T])$. Then
\begin{equation*}
w''+\lambda_1w(t)=\int_{\Omega} {_0I_t^{\gamma}}(|u|^{p})\varphi_1dx
\end{equation*}
in the sense of distributions. Furthermore, since $w\in C([0,T])$ and $\int_{\Omega} {_0I_t^{\gamma}}(|u|^{p})\varphi_1dx \in C([0,T])$, we deduce from the regularity theory that $w\in C^2([0,T])$. Then in terms of Jensen's inequality, we have
\[
w''+\lambda_1w(t)
\geq {_0I_t^{\gamma}} |w|^p  ,\ \ t\in [0,T],
\]
which implies $T<+\infty$ by Corollary \ref{coro2}(iv).
\end{proof}
\begin{remark}\label{re3-10}
By  Corollary \ref{coro1} (i) and Corollary \ref{coro2}(i), we know that for given $T>0$, if
\[
T\int_{\Omega}u_1(x)\varphi_1(x)dx+\int_{\Omega}u_0(x)\varphi_1(x)dx
>K_1T^{\alpha+\gamma-\frac{p\gamma}{p-1}}
      +K_2T^{-\frac{\alpha+\gamma}{p-1}},
\]
then the corresponding solution of \eqref{20.1} does not exist globally in time and $T^*<T$.
\end{remark}

Finally, we have the following results of global existence of solutions for sufficiently small initial values.
\begin{theorem}\label{th6.30.1}
Let $p>1$, $1<\alpha<2$ and $u_0,u_1\in L^\infty(\Omega)$.
\begin{enumerate}
\item[\rm{(i)}] If $\alpha+\gamma\geq 2$, $p(1-\gamma)> 1$ and $\|u_0\|_{L^\infty(\Omega)}+\|u_1\|_{L^\infty(\Omega)}$ is sufficiently small, then the mild solution $u$ of \eqref{20.1} is global.
\item[\rm{(ii)}] If $\alpha+\gamma<2$, $p(1-\gamma)> 1$, $u_1\equiv0$ and $\|u_0\|_{L^\infty(\Omega)}$ is sufficiently small, then the mild solution of \eqref{20.1} exists globally.
\item[\rm{(iii)}] If $\alpha+\gamma<2$, $p\geq 1+\frac{\gamma}{\alpha-1}$ and $\|u_0\|_{L^\infty(\Omega)}+\|u_1\|_{L^\infty(\Omega)}$ is sufficiently small, then the mild solution of \eqref{20.1} exists globally.
\end{enumerate}
\end{theorem}
\begin{proof}
(i) We prove the global existence of solutions for problem
\eqref{20.1} by the contraction mapping principle.

Let
$
X=\{u\in L^{\infty}((0,\infty),L^\infty(\Omega))\ |\
\|u\|<\infty\},
$
where
$
\|u\|=\sup_{t>0}(1+t)^\frac{\gamma}{p-1}\|u(t)\|_{L^\infty(\Omega)}.
$
Then $X$ is a Banach space.
For given $u\in X,$ we define
\[
\Psi(u)(t)=P_\alpha(t)u_0+{}_0I_t^1P_\alpha(t)u_1+
\int_0^t(t-s)^{\alpha-1}S_\alpha(t-s){_0I_s^{\gamma}}(|u|^{p})ds,\ t\geq0.
\]
Let $B_M$
denote the closed ball in $X$ with center $0$ and radius $M$,
where $M>0$ is to be chosen sufficiently small.

To prove our result, it suffices to show that $\Psi$ is a contractive mapping on $\mathcal {E}$ when $\|u_0\|_{L^\infty(\Omega)}$, $\|u_1\|_{L^\infty(\Omega)}$ and $M$ are chosen sufficiently small. The assumptions that $\alpha+\gamma\geq 2$ and $p>\frac{1}{1-\gamma}$ imply $p>\frac{1}{1-\gamma}\geq1+\frac{\gamma}{\alpha-1}$, $\frac{\gamma}{p-1}<\frac{\gamma}{p-1}+1\leq\alpha$ and $\frac{p\gamma}{p-1}<1$.
Hence, by Lemma \ref{lemma100},  there
exists a constant $C>0$ such that
for any $u\in B_M$  and $t\geq0,$
\begin{align}\label{9.13.1}
(1+t)^\frac{\gamma}{p-1} \|P_\alpha(t)u_0\|_{L^\infty(\Omega)}
\leq C(1+t)^{\frac{\gamma}{p-1}-\alpha}\|u_0\|_{L^\infty(\Omega)}\leq C\|u_0\|_{L^\infty(\Omega)},
\end{align}
\begin{align}\label{7.1.1}
(1+t)^\frac{\gamma}{p-1} \|{}_0I_t^1P_\alpha(t)u_1\|_{L^\infty(\Omega)}
\leq C(1+t)^{\frac{\gamma}{p-1}+1-\alpha}\|u_1\|_{L^\infty(\Omega)}\leq C\|u_1\|_{L^\infty(\Omega)},
\end{align}
and
\begin{align}\label{9.13.2}
&(1+t)^\frac{\gamma}{p-1}\|\Psi(u)-P_\alpha(t)u_0-{}_0I_t^1P_\alpha(t)u_1\|_{L^\infty(\Omega)}\nonumber\\
\leq&C(1+t)^\frac{\gamma}{p-1}\int_0^t\int_0^s
\frac{(t-s)^{\alpha-1}(s-\tau)^{\gamma-1}}{1+(t-s)^{2\alpha}}
\|u(\tau)\|_{L^\infty(\Omega)}^pd\tau ds\nonumber\\
\leq&CM^p(1+t)^\frac{\gamma}{p-1}
\int_0^t\int_0^s\frac{(t-s)^{\alpha-1}(s-\tau)^{\gamma-1}}{1+(t-s)^{2\alpha}}
(1+\tau)^{-\frac{p\gamma}{p-1}}d\tau ds\nonumber\\
\leq& CM^p(1+t)^\frac{\gamma}{p-1}
\int_0^t\frac{(t-s)^{\alpha-1}}{1+(t-s)^{2\alpha}}
\int_0^s(s-\tau)^{\gamma-1}\tau^{-\frac{p\gamma}{p-1}}d\tau ds\nonumber\\
=&CM^p(1+t)^\frac{\gamma}{p-1}
\int_0^t\frac{(t-s)^{\alpha-1}}{1+(t-s)^{2\alpha}}
s^{-\frac{\gamma}{p-1}}ds.
\end{align}

Next, we estimate \eqref{9.13.2}.
In terms of \eqref{100.1}, we know that there exist positive constants $C_1,C_2$ and $L$ such that for $t>L$
\[
\frac{C_1}{1+t^{2\alpha}}\leq -E_{\alpha,\alpha}(-t^\alpha)\leq \frac{C_2}{1+t^{2\alpha}}.
\]
This implies that for $t>2L$
\begin{align}\label{7.29.1}
&\int_0^t\frac{(t-s)^{\alpha-1}}{1+(t-s)^{2\alpha}}
s^{-\frac{\gamma}{p-1}}ds\nonumber\\
\leq& C\int_0^{t-L}-(t-s)^{\alpha-1}E_{\alpha,\alpha}(-(t-s)^\alpha)s^{-\frac{\gamma}{p-1}}ds
+\int_{t-L}^t\frac{(t-s)^{\alpha-1}}{1+(t-s)^{2\alpha}}
s^{-\frac{\gamma}{p-1}}ds\nonumber\\
=&C\int_0^{t}-(t-s)^{\alpha-1}E_{\alpha,\alpha}(-(t-s)^\alpha)s^{-\frac{\gamma}{p-1}}ds
+C\int_{t-L}^{t}(t-s)^{\alpha-1}
E_{\alpha,\alpha}(-(t-s)^\alpha)s^{-\frac{\gamma}{p-1}}ds\nonumber\\
&+\int_{t-L}^t\frac{(t-s)^{\alpha-1}}{1+(t-s)^{2\alpha}}
s^{-\frac{\gamma}{p-1}}ds.
\end{align}
Furthermore, since $\frac{\gamma}{p-1}<1$ and $\alpha>1$, we know  that $(t-s)^{\alpha-1}E_{\alpha,\alpha}((t-s)^\alpha)
s^{-\frac{\gamma}{p-1}}\in L^1((0,t))$ for given $t>0$. Thus it follows from the dominated convergence theorem that
\begin{align*}
\int_0^t(t-s)^{\alpha-1}E_{\alpha,\alpha}(-(t-s)^\alpha)
s^{-\frac{\gamma}{p-1}}ds
&=\sum_{k=0}^\infty\int_0^t\frac{(-1)^k(t-s)^{\alpha k+\alpha-1}s^{-\frac{\gamma}{p-1}}}{\Gamma(\alpha k+\alpha)}ds\\
&=\Gamma(1-\frac{\gamma}{p-1})t^{\alpha-\frac{\gamma}{p-1}}
E_{\alpha,\alpha+1-\frac{\gamma}{p-1}}(-t^\alpha).
\end{align*}
 Consequently, we conclude from \eqref{7.29.1} that for $t>2L$
\[
\int_0^t\frac{(t-s)^{\alpha-1}}{1+(t-s)^{2\alpha}}
s^{-\frac{\gamma}{p-1}}ds\leq -Ct^{\alpha-\frac{\gamma}{p-1}}
E_{\alpha,\alpha+1-\frac{\gamma}{p-1}}(-t^\alpha)+Ct^{-\frac{\gamma}{p-1}}\\
\leq Ct^{-\frac{\gamma}{p-1}}.
\]
Hence
\begin{equation}\label{7.4.1}
(1+t)^\frac{\gamma}{p-1}\|\Psi(u)-P_\alpha(t)u_0-{}_0I_t^1P_\alpha(t)u_1\|_{L^\infty(\Omega)}
\leq CM^p.
\end{equation}

On the other hand, an argument similar to the above proof shows that there exists a constant $C>0$ such that for any $u,v\in B_M$  and $t\geq0,$
\begin{align}\label{9.13.3}
&(1+t)^\frac{\gamma}{p-1}\|\Psi(u)-\Psi(v)\|_{L^\infty(\Omega)}\nonumber\\
\leq& CM^{p-1}(1+t)^\frac{\gamma}{p-1}
\int_0^t\frac{(t-s)^{\alpha-1}}{1+(t-s)^{2\alpha}}
\int_0^s(s-\tau)^{\gamma-1}\tau^{-\frac{p\gamma}{p-1}}d\tau ds\|u-v\|\nonumber\\
\leq&CM^{p-1}(1+t)^\frac{\gamma}{p-1}
\int_0^t\frac{(t-s)^{\alpha-1}}{1+(t-s)^{2\alpha}}
s^{-\frac{\gamma}{p-1}}ds\|u-v\|\nonumber\\
\leq &CM^{p-1}\|u-v\|.
\end{align}

Combining \eqref{9.13.1},\eqref{7.1.1}, \eqref{7.4.1} and \eqref{9.13.3}, we know that $\Psi$ is a strict contractive map on $B_M$ if $\|u_0\|_{L^\infty(\Omega)}$, $\|u_1\|_{L^\infty(\Omega)}$ and $M$ are chosen small enough. Then the contraction mapping principle implies that $\Psi$ has a unique
fixed point $u\in B_M$. In addition, from Lemma \ref{lemma100},  we know that $u\in C((0,+\infty),L^{\infty}(\Omega))$ and $\lim_{t\rightarrow 0^+}\|u(t)-P_\alpha(t)u_0\|_{L^\infty(\Omega)}=0$. Thus problem \eqref{20.1} admits a global mild solution.

(ii) Since $p>\frac{1}{1-\gamma}$ implies that  $\frac{\gamma}{p-1}<\frac{p\gamma}{p-1}<1<\alpha$, we know that the estimates \eqref{9.13.1},\eqref{9.13.2}, \eqref{7.4.1} and \eqref{9.13.3} also hold in this case.  Thus we can obtain the desired conclusion.

(iii) The assumption that $\alpha+\gamma<2$ and $p\geq 1+\frac{\gamma}{\alpha-1}$ imply that  $p\geq1+\frac{\gamma}{\alpha-1}>\frac{1}{1-\gamma}$, $\frac{\gamma}{p-1}<\frac{\gamma}{p-1}+1\leq\alpha$ and $\frac{p\gamma}{p-1}<1$.
Then, repeating the arguments in the proof of case (i), we can complete the proof.
\end{proof}

\section*{Acknowledgment}
\noindent

This work was supported in part by Young Backbone Teachers of Henan Province(No.2021GGJS130).


\begin{thebibliography}{99}
\bibitem{acv} M. Allen, L. Caffarelli, A.Vasseur, A parabolic problem with a fractional time derivative, Arch. Ration. Mech. Anal. 221 (2016), 603-630.

\bibitem{agkw} E. Alvarez, C.G. Gal, V. Keyantuo, M. Warma, Well-posedness results for a class of
semi-linear super-diffusive equations, Nonlinear Anal. TMA 181 (2019), 24-61.

\bibitem{andrade} B. Andrade, T.S. Cruz, Regularity theory for a nonlinear fractional reaction-diffusion equation, Nonlinear Anal. 195 (2020), Article 111705.

\bibitem{asofwa} S.A. Asogwa, J.B. Mijena, E. Nane, Blow-up results for space-time fractional stochastic partial differential equations, Potential Anal. 53 (2020), 357-386.

\bibitem{T.Cazenave} T. Cazenave, F. Dickstein, F.B. Weissler, An
equation whose Fujita critical exponent is not given by scaling,
Nonlinear Anal. 68 (2008), 862-874.

\bibitem{chen} W. Chen, Interplay effcts on blow-up of weakly coupled systems for
semilinear wave equations with general nonlinear memory terms, Nonlinear Anal. 202 (2021), Article 112160.

\bibitem{chen2} W. Chen, A. Palmieri, in: M. Cicognani, D. Del Santo, A. Parmeggiani, M. Reissig (Eds.), Blow-up Result for a Semilinear Wave Equation with a Nonlinear Memory Term, in: Springer INdAM Series, vol. 43, 2020, p. 20.

\bibitem{Abbicco} M. D'Abbicco, The influence of a nonlinear memory on the damped wave equation, Nonlinear Anal. 95 (2014), 130-145.

\bibitem{fino2}A.Z. Fino, Critical exponent for damped wave equations with nonlinear memory. Nonlinear Anal. 74(16) (2011), 5495-5505.

\bibitem{fino1}A.Z. Fino, M. Jazar, Blow-up solutions of second-order differential inequalities with a nonlinear memory term, Nonlinear Anal. 75(6)(2012), 3122-3129.

\bibitem{folland}G.B. Folland, Real Analysis: Modern Techniques and Their Applications, Wiley, New York, 1999.

\bibitem{giga}Y. Giga, T. Namba, Well-posedness of Hamilton-Jacobi equations with Caputo's time fractional derivative, Commun. Partial Differ. Equ., 42 (2017), 1088-1120.

\bibitem{Gilbarg}D. Gilbarg, N.S. Trudinger, Elliptic Partial Differential Equations of Second Order, Springer-Verlag, Berlin, 1998.

\bibitem{Ka}S. Kaplan, On the growth of solutions of quasi-linear parabolic equations, Comm. Pure Appl. Math. 16 (1963), 305-330.

\bibitem{kilbas2006} A.A. Kilbas, H.M. Srivastava, J.J. Trujillo, Theory and
Applications of Fractional Differential Equations, vol. 204. Elsevier
Science B.V., Amsterdam, 2006.

\bibitem{kim} I. Kim, K.H. Kim, S. Lim, An $L_q(L_p)$-theory for the time fractional evolution equations with variable coefficients, Adv. Math. 306 (2017), 123-176.
    
\bibitem{Ki} M. Kirane, Y. Laskri, N.E. Tatar, Critical exponents of Fujita type for certain evolution equations and systems with spatio-temporal fractional derivatives, J. Math. Anal. Appl. 312 (2005), 488-501.

\bibitem{llz} N.A Lai J.L. Liu, J.L Zhao, Blow up for Initial-Boundary Value Problem of Wave
Equation with a Nonlinear Memory in 1-D, Chin. Ann. Math.
38B(3) (2017), 827-838.

\bibitem{llw} L. Li, J.G. Liu, L. Wang, Cauchy problems for Keller-Segel type time-space fractional diffusion equation, J. Differ. Equ., 265 (2018), 1044-1096.

\bibitem{li} Y. Li, G. Zhang, Blow-up and global existence of solutions for a time fractional diffusion equation, Frac. Calc. Appl. Anal. 21(2018), 1619-1640.

\bibitem{Metzler1} R. Metzler, J. Klafter, The random walk¡¯s guide to anomalous diffusion: a fractional dynamics
approach, Phys. Rep. 339 (2000), 1-77.

\bibitem{podlubny1999} I. Podlubny, Fractional Differential
Equations, Academic Press, New York, 1999.

\bibitem{Pskhu} A.V. Pskhu, On the real zeros of functions of Mittag-Leffler type, Math. Notes, 77 (2005), 546-552.

\bibitem{qs} P. Quittner, P. Souplet, Superlinear Parabolic Problems: Blow-up, Global Existence and Steady States, Birkh\"{a}user, Basel, 2007.


\bibitem{Schneider} W.R. Schneider, W. Wyss, Fractional diffusion and wave equations, J. Math. Phy. 30(1989),134-144.

\bibitem{tuan} N.H. Tuan, V.V. Au, R. Xu, Semilinear Caputo time-fractional pseudo-parabolic equations, Commun. Pure Appl. Anal. 20 (2021), 583-621.

\bibitem{vr} V. Vergara, R. Zacher, Stability, instability, and blowup for time fractional and other nonlocal in time semilinear subdiffusion equations, J. Evol. Equ. 17 (2017), 599-626.


\bibitem{R.N.Wang} R.N. Wang, D.H. Chen, T.J. Xiao, Abstract
fractional Cauchy problems with almost sectorial operators, J.
Differ. Equ. 252 (2012), 202-235.

\bibitem{wzf} J.R. Wang, Y. Zhou, M. Fe\u{c}kan, Abstract Cauchy problem for fractional differential equations, Nonlinear Dyn. 71 (2013), 685-700.

\bibitem{Yordanov} B.T. Yordanov, Q.S. Zhang, Finite time blow-up for critical wave equations in high dimensions, J. Funct. Anal. 231 (2006), 361-374.

\bibitem{zacher} R. Zacher, A De Giorgi-Nash type theorem for time fractional diffusion equations, Math. Ann. 356 (2013), 99-146.


\bibitem{zhangli2}Q.G. Zhang, Y.N. Li, The critical exponent for a time fractional diffusion equation with nonlinear memory, Math. Meth. Appl. Sci. 41 (2018), 6443-6456.

\bibitem{zhangli3}Q.G. Zhang, Y.N. Li, The critical exponents for a time fractional diffsion equation with nonlinear memory in a bounded domain, Appl. Math. Lett. 92 (2019), 1-7.


\bibitem{Zhangli4} Q.G. Zhang, Y.N. Li, Global well-posedness and blow-up solutions of the Cauchy problem for a time-fractional superdiffusion equation,  J. Evol. Equ. 19 (2019), 271-303.


\bibitem{ZhangSun}Q.G. Zhang, H.R. Sun, The blow-up and global existence of solutions of Cauchy problems for a time fractional diffusion equation, Topol. Meth. Nonlinear Anal. 46(1) (2015), 69-92.

\bibitem{zhou} Y. Zhou, J.W. He, Well-posedness and regularity for fractional damped wave equations, Monatsh.  Math. 194 (2021), 1-34.
    


\end{thebibliography}
\end{document}